\documentclass[11pt]{amsart}

\usepackage[left=2.5cm, right=2.5cm, top=2cm, bottom=2cm]{geometry}

\usepackage[utf8]{inputenc}
\usepackage[english]{babel}
\usepackage{enumitem}
 
\usepackage{amsthm}
\usepackage{amsmath}
\usepackage{amssymb}

\usepackage{graphicx}
\usepackage{caption}
\usepackage{subcaption}

\newtheorem{theorem}{Theorem}

\newtheorem{lemma}{Lemma}
\newtheorem{prop}[theorem]{Proposition}

\theoremstyle{remark}

\newtheorem{remark}{Remark}

\renewcommand{\leq}{\leqslant}
\renewcommand{\geq}{\geqslant}
\usepackage{setspace}
\usepackage{float}

\title{Character sums over products of prime polynomials}
\author[S. Porritt]{Sam Porritt}
\address{Department of Mathematics\\University College London\\
	25 Gordon Street, London, England}
\email{porritt.samuel@gmail.com}
\begin{document}
\onehalfspacing
\maketitle
\vspace{-0.8cm}
\begin{abstract}
We study sums of Dirichlet characters over polynomials in $\mathbb{F}_q[t]$ with a prescribed number of irreducible factors. Our main results are explicit formulae for these sums in terms of zeros of Dirichlet $L$-functions. We also exhibit new phenomena concerning Chebyshev-type biases of such sums when the number of irreducible factors is very large.
\end{abstract}
\vspace{-0.2cm}
\section{Introduction}
\subsection{Set-up}
Let $\mathbb{F}_q[t]$ be the polynomial ring in one variable over the finite field $\mathbb{F}_q$ and $\mathcal{M}, \mathcal{P}$ be the subsets of monic and prime monic polynomials, respectively. Let $\mathcal{M}_n$ be the set of monic polynomials of degree $n$. This paper concerns the sums of non-trivial Dirichlet characters $\chi : \mathbb{F}_q[t] \rightarrow \mathbb{C}$
$$\pi_{k}(n,\chi)=\sum_{\substack{f \in \mathcal{M}_n \\ \Omega(f)=k }}\chi(f),$$
where $\Omega(f)$ is the number of prime divisors of $f$ counted with multiplicity. We shall use analytic arguments to relate the quantity $\pi_{k}(n,\chi)$ to the zeros of $L(u,\chi)$, the $L$-function associated to $\chi$, which is defined as $$L(u,\chi):=\sum_{f \in \mathcal{M}}\chi(f)u^{\deg f} = \prod_{p \in \mathcal{P}}(1-\chi(p)u^{\deg p})^{-1}.$$
It is known that the analogue of the Riemann hypotheis holds for such $L$-functions and in particular, using the same notation as~\cite[Equation (1)]{L-M}, we can factor the polynomial $L(u,\chi)$ as a product of linear factors
\begin{equation}
L(u,\chi)=:(1-\sqrt{q}u)^{m_{+}}(1+\sqrt{q}u)^{m_{-}}\prod_{j=1}^{d_\chi}(1-\alpha_j(\chi)u)^{m_j}\prod_{j'=1}^{d'_\chi}(1-\beta_{j'}(\chi)u)
\end{equation}
where $|\beta_{j'}(\chi)|=1$ and $\alpha_j(\chi) = \sqrt{q}e^{i\gamma_j(\chi)}$ is non-real and has absolute value $\sqrt{q}$. See for example, \cite[Proposition 4.3]{Rosen}. The $\alpha_j$ are distinct for distinct $j$ and appear with multiplicity $m_j$. For our purposes, the $\beta_{j'}$ are less important. We are interested in uniformity with respect to the variables $n$ and $k$ so, everywhere apart from equation (2) below, implied constants may depend on anything except $n$ and $k$ (in particular, on $\chi$ and $q$). It is convenient to use the normalisation
$$\widetilde{\pi_{k}}(n,\chi):=\pi_{k}(n,\chi)\frac{n(k-1)!(-1)^k}{q^{n/2}(\log n)^{k-1}}.$$

\subsection{Brief background on Chebyshev's bias}
Assume for the time being that $L(\pm q^{-1/2},\chi) \neq 0$ and that each zero of $L(u, \chi)$ is simple. The corresponding assumption, that $L(1/2,\chi) \neq 0$ and all zeros are simple, is conjectured to hold for all number field Dirichlet $L$-functions. Then it follows from the work of Devin and Meng~\cite{L-M}, and Theorems~\ref{thm4.1} and~\ref{thm4.3} below, that for each \emph{fixed} $k$,
\begin{equation}
\widetilde{\pi_{k}}(n,\chi) = \textbf{1}_{\chi^2 = \chi_0}\frac{1+(-1)^n}{2^k}+\sum_{\gamma_j \neq 0, \pi}e^{i n \gamma_j} +O_k\left(\frac{1}{\log n}\right).
\end{equation}
The terms in this formula corresponding to non-real zeros of $L(u,\chi)$ oscillate around 0 as $n$ increases. We can think of the other term as a `bias' term which biases $\pi_k(n,\chi)$ away from 0. It follows that, if $\chi$ is real then for each fixed $k$, the quantity $(-1)^k\pi_k(n,\chi)$ is biased towards being positive rather than negative, but that ``as
$k$ increases, the biases become smaller and smaller"--\cite{L-M}. The results of this paper include and generalise~(2) to $k(n)$ varying with $n$. We shall see that the bias term does indeed tend to $0$ for certain sequences $k(n) \rightarrow \infty$, but, perhaps surprisingly, not if $k(n)$ grows too quickly. In particular, there is a constant $\gamma \approx 1.2021\ldots$ such that if $\gamma < k/\log n < \sqrt{q}$, then the bias term is larger than the oscillating terms for every large even $n$.

In the case that $k=1$, the integer analogue of the explicit formula (2) has a corresponding bias term that is responsible for the phenomena known as \emph{Chebyshev's bias}, named after Chebyshev who observed in 1853 that ``There is a notable difference in the splitting of the prime numbers between the
two forms $4n + 3$ and $4n + 1$: the first form contains a lot more than the second". Chebyshev's observation was formalised by Rubenstein and Sarnak~\cite{R-S} who proved, under certain conjectures on the zeros of Dirichlet $L$-functions, that, if 
$$S_k = \Big\{ x \in \mathbb{N} \: : \: (-1)^k\sum_{\substack{n \leq x \\ \Omega(n) = k}}\chi(n) > 0 \Big\}$$
where $\chi$ is the non-trivial Dirichlet character mod $4$, then the set $S_1$ has logarithmic density $\approx 0.996$. Figure 1. shows the plots of the normalised character sum
$$\widetilde{\pi_k^{\text{int}}}(x) = \frac{(-1)^k(k-1)!\log x}{\sqrt{x} (\log\log x)^{k-1}} \sum_{\substack{n \leq x \\ \Omega(n) = k}}\chi(n).$$
\begin{figure}[H]
	\centering
	\begin{subfigure}{.5\textwidth}
		\centering
		\includegraphics[scale=0.26]{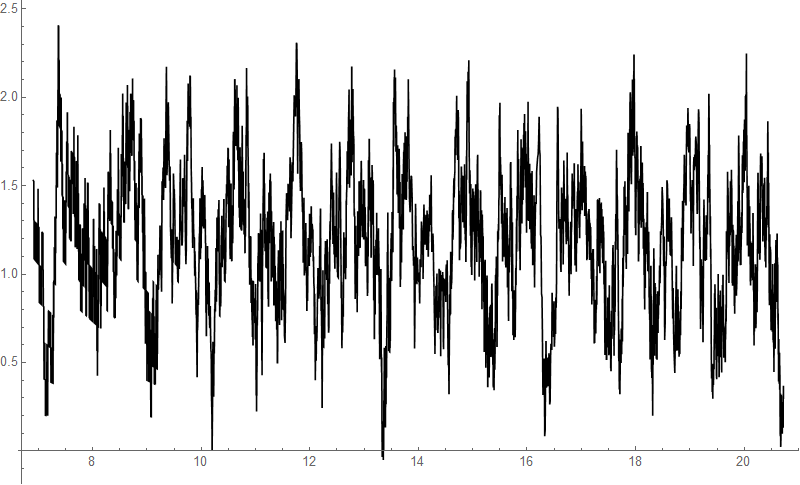}
		\caption{$k=1$}
	\end{subfigure}%
	\begin{subfigure}{.5\textwidth}
		\centering
		\includegraphics[scale=0.26]{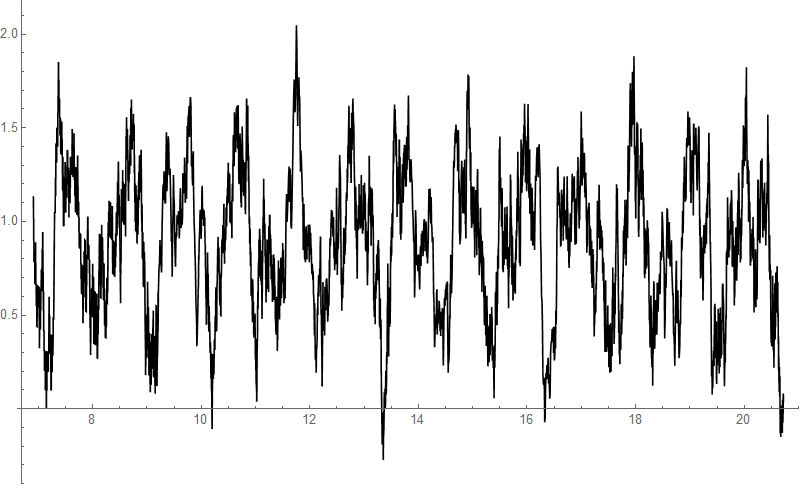}
		\caption{$k=2$}
	\end{subfigure}\\
	\begin{subfigure}{.5\textwidth}
		\centering
		\includegraphics[scale=0.26]{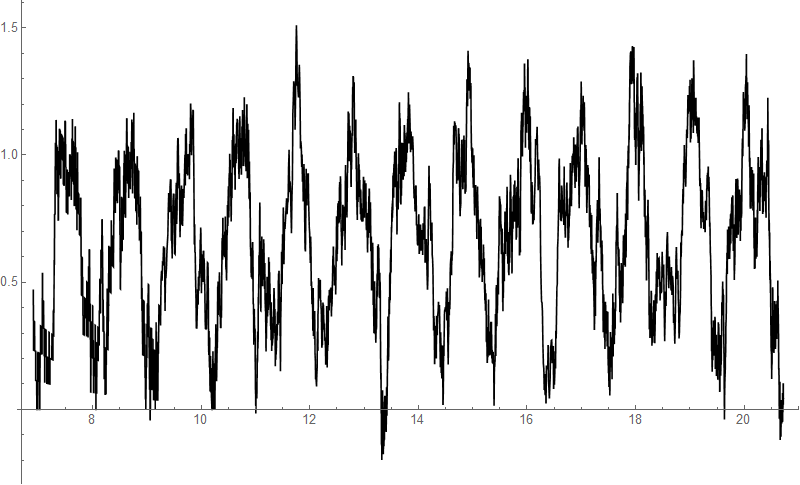}
		\caption{$k=3$}
	\end{subfigure}%
	\begin{subfigure}{.5\textwidth}
		\centering
		\includegraphics[scale=0.26]{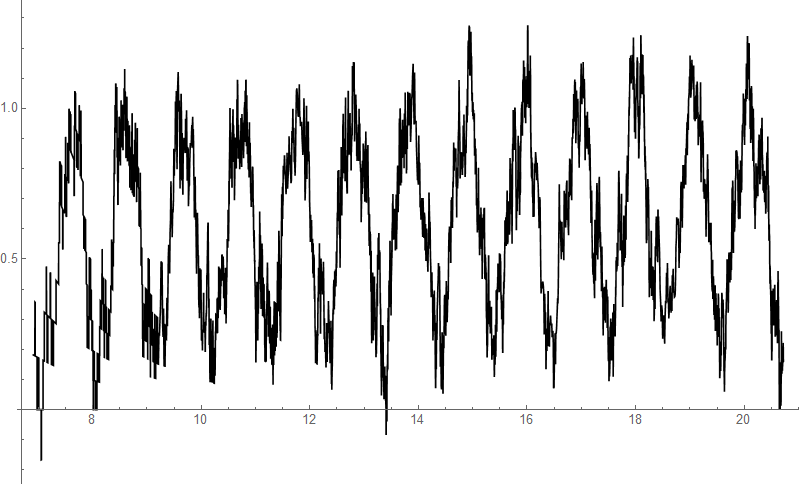}
		\caption{$k=4$}
	\end{subfigure}
	\caption{Plots of $\widetilde{\pi_k^{\text{int}}}(x)$ for $x$ up to $10^9$ plotted on a logarithmic scale}
\end{figure}
The phenomena has since been extensively studied and generalised in a number of directions. See~\cite{G-M} for an introduction to the topic. Ford and Sneed~\cite{F-S} proved, again under certain conjectures on the zeros of $L(s,\chi)$, that the set $S_2$ has logarithmic density $\approx 0.894$. A more general theorem capturing the Chebyshev bias for fixed $k \geq 1$ was proved for integers by Meng~\cite{M} and for polynomials by Devin and Meng~\cite{L-M}. For example, it follows from~\cite{M} and standard conjectures on the zeros of Dirichlet $L$-functions that the set $S_k$ has logarithmic density $\delta_k$ where $1/2< \delta_k < 1$ and $\delta_k \rightarrow 1/2$ as $k \rightarrow \infty$. This suggests that in some sense the bias dissipates as $k \rightarrow \infty$. However, we shall see that, for polynomials and $k(n)$ increasing with $n$, sometimes the bias is strong enough to ensure that $(-1)^k\pi_k(n,\chi) >0 $ for all large even $n$.

From now on, $\chi$ will always denote a non-principal Dirichlet character modulo a fixed polynomial $d \in \mathbb{F}_q[t]$ and $\chi_0$ will denote the principal character mod $d$. It may be helpful to bear in mind that for large values of $n$, the distribution of $\Omega(f)$, given a polynomial $f$ selected uniformly at random from $\mathcal{M}_n$, is approximately normal with mean $\log n$ and standard deviation $\sqrt{\log n}$.

Because the behaviour of $\pi_k(n, \chi)$ depends on whether or not $\chi$ is real, we shall present the results for the two cases separately.

\subsection{When $\chi$ is not real}

\begin{theorem}\label{thm4.1}
	Suppose $\chi^2 \neq \chi_0$ and let $\epsilon >0$. With the same notation as in equation (1),
	$$
	\widetilde{\pi_k}(n,\chi)= \sum_{\gamma_j\neq 0, \pi}m_j^ke^{i n \gamma_j} + m_{+}^k +(-1)^n m_{-}^k
	+  O\left( \frac{k}{\log n}\left(\sum_{\gamma_j \neq 0, \pi}m_j^k + m_{+}^k + m_{-}^k \right) \right)
	$$
	uniformly for $ 1\leq k \leq q^{1/2-\epsilon} \log n$.
\end{theorem}
When $k$ is large with respect to $n$, this improves upon \cite[Theorem 5.1]{L-M}, which requires that $k = o(\sqrt{\log n})$ and has an error term involving a factor $(\deg d)^k$.

Theorem~\ref{thm4.1} gives a main term and smaller error term in the range $k = o(\log n)$. A more general formula that describes the behaviour of $\widetilde{\pi_k}(n,\chi)$ in the full range $1\leq k\leq q^{1/2-\epsilon}\log n$ can be extracted from the proof but is significantly more complicated to write down. If we assume certain simplifying assumptions, though, we can state the more general formula as follows.
\begin{theorem}\label{thm4.2}
	Suppose that $\chi^2 \neq \chi_0$ and that $m_j = 1$ for each $j$ and $m_{\pm}=0$. Suppose $k(n)$ is a sequence such that $\alpha = \lim_{n \rightarrow \infty}\frac{k}{\log n}$ exists and $0 \leq \alpha <q^{1/2}$.
	Then there exist non-zero constants $h_j(\alpha)$ such that
	$$\widetilde{\pi_k}(n,\chi) = \sum_{\gamma_j \neq 0, \pi}h_j(\alpha) e^{i n \gamma_j} + o(1).$$
\end{theorem}
Recall that the assumptions $m_j=1$ for each $j$ and $m_{\pm}=0$ is conjectured to hold for number field Dirichlet $L$-functions.
\begin{remark}
	The coefficients $h_j(\alpha)$ are explicitly defined in terms of $\alpha$ and $\chi$ in the course of the proof but are quite lengthy to write down in full. An even more general formula that does not require the limit $\lim_{n \rightarrow \infty}\frac{k}{\log n}$ to exist can be extracted from the proof but is more complicated to write down.
\end{remark}

\subsection{When $\chi$ is real}

The next two theorems show how the behaviour of $\pi_k(n,\chi)$ differs significantly when $\chi$ is real. The first deals with `small' values of $k$ and again extends the range of a formula given in~\cite{L-M}.
\begin{theorem}\label{thm4.3}
	Suppose $\chi^2 = \chi_0$. With the same notation as in equation (1), uniformly for $1\leq k \leq (\log n)^{1/2}$ we have
	\begin{multline*}
	\widetilde{\pi_k}(n,\chi)= \sum_{\gamma_j\neq 0, \pi}m_j^ke^{i n \gamma_j} + (m_{+}+1/2)^k+(-1)^n(m_{-}+1/2)^k \\
	+  O\left(\tfrac{k}{\log n}\sum_{\gamma_j \neq 0, \pi}m_j^k + \tfrac{k^2}{\log n}\max_{\pm}\left\{(m_{\pm}+1/2)^k \right\} \right).
	\end{multline*}
	Moreover, uniformly in the range $1\leq k \leq (\log n)^{2/3}$ we have
	\begin{multline*}
	\widetilde{\pi_k}(n,\chi)= \sum_{\gamma_j\neq 0, \pi}m_j^ke^{i n \gamma_j}
	+ (m_{+}+1/2)^ke^{\tfrac{(k-1)^2}{2(m_{+}+1/2)^2\log n}}+(-1)^n(m_{-}+1/2)^ke^{\tfrac{(k-1)^2}{2(m_{-}+1/2)^2\log n}} \\
	+  O\left( \tfrac{k}{\log n}\sum_{\gamma_j \neq 0, \pi}m_j^k + \left(\tfrac{1}{k}+\tfrac{k^3}{(\log n)^2}\right)\max_{\pm}\left\{ (m_{\pm}+1/2)^ke^{\tfrac{(k-1)^2}{2(m_{\pm}+1/2)^2\log n}} \right\} \right).
	\end{multline*}
\end{theorem}
Notice that the error terms in the formulae above are essentially the same as the corresponding main terms but with certain extra factors involving $k$ and $\log n$. It is not too difficult to see that these extra factors are $o(1)$ in the first formula if $k = o(\sqrt{\log n})$. They are $o(1)$ in the second if $k \rightarrow \infty$ and $k = o((\log n)^{2/3})$. The change in behaviour and appearance of extra terms for $k$ around $\sqrt{\log n}$ may explain why~\cite[Theorem 5.1]{L-M} required $k = o(\sqrt{\log n}).$ The significance of the threshold $k$ around $(\log n)^{2/3}$ will become apparent from the proof.

If we make the same simplifying assumptions as in Theorem~\ref{thm4.2}, that is, the limit $ \lim_{n\rightarrow} \frac{k}{\log n}$ exists and $m_j = 1$, ${m_{\pm}} = 0$, then we get a corresponding version for real $\chi$ that includes the bias term. The size of the bias term can be described in terms of the continuous function $b(\alpha)$ defined by
\begin{equation}
b(\alpha) = \alpha\left(\frac{s(\alpha)-1}{2}-\log(2s(\alpha))\right), \hspace{0.4cm} \text{ where } \hspace{0.4cm} s(\alpha) = \frac{1}{8\alpha}\left(\sqrt{1+16\alpha} - 1\right).
\end{equation}
\begin{figure}[H]
	{
		\centering
		\includegraphics[scale=0.45]{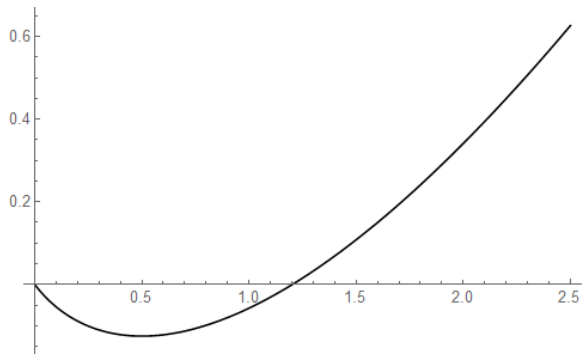}
		\caption{Plot of $b(\alpha)$}
	}
\end{figure}
Of particular significance is the fact that, if the real constants $\beta \approx 0.3637\ldots$ and $\gamma \approx 1.2021\ldots$ are defined by the two equations $$e^{\beta-1}=4\beta^2  \hspace{0.5cm} \text{ and } \hspace{0.5cm}  \gamma = \frac{1-\beta}{4\beta^2},$$
then
$b(\alpha)<0$ for $0 < \alpha < \gamma$ and $b(\alpha) > 0$ for $ \alpha > \gamma$.
\begin{theorem}\label{thm4.4}
	Suppose $\chi^2 = \chi_0$ and that $m_j = 1$ for each $j$ and $m_{\pm}=0$. Suppose $k(n)$ is a sequence such that $\alpha = \lim_{n \rightarrow \infty}\frac{k}{\log n}$ exists and $0 < \alpha <q^{1/2}$.
	Then there exist non-zero constants $h_j(\alpha)$ such that
	$$\widetilde{\pi_k}(n,\chi) = \left\{h_{+}(\alpha)+(-1)^n h_{-}(\alpha) + o(1)\right\}n^{b((k-1)/\log n)} + \sum_{\gamma_j \neq 0, \pi}h_j(\alpha) e^{i n \gamma_j} + o(1)$$
	where
	$$h_{\pm}(\alpha) = P_{\pm}(r)\frac{2^{-r(r + 1)/2}}{\Gamma(r(r + 1)/2)}\frac{L(\pm q^{-1/2},\chi)^{-r}}{\sqrt{1+r^2/\alpha}}$$
	$$P_{\pm}(r) = \prod_{p \in \mathcal{P}}\left(1+\frac{r \chi(p)(\pm 1)^{\deg p}}{q^{\deg p/2}}\right)^{-1}\left(1-\frac{\chi(p)(\pm1)^{\deg p}}{q^{\deg p /2}}\right)^{-r}\left(1-\frac{1}{q^{\deg p}}\right)^{r(r+1)/2}$$
	and $r=r(\alpha)$ is the positive root of
	$$r^2 + \frac{r}{2} - \alpha = 0.$$
\end{theorem}
\begin{remark}
	It follows that the oscillating terms dominate if $0 < \alpha <\gamma$, but that the bias terms dominate if $\gamma < \alpha <q^{1/2}$, at least if $n$ is even. This is perhaps surprising given the results in~\cite{M} and~\cite{L-M}, which, as explained in section 1.1 above, suggest that in some sense ``as
	$k$ increases, the biases become smaller and smaller". Since the conditions we assumed on the zeros of $L(u,\chi)$ are conjectured to hold for number field Dirichlet $L$-functions, we wonder whether something analogous happens in the number field setting. 
\end{remark}
\begin{remark}
	The remark following Theorem~\ref{thm4.2} concerning the coefficients $h_j(\alpha)$ applies verbatim to the coefficients $h_j(\alpha)$ in Theorem 4.
\end{remark}

\section{A Selberg-Delange type argument}

Let $A>0$ and $u$ and $z$ be complex variables satisfying $|z|\leq A$ and $|zu|<1$. Deploying the Selberg-Delange method to study the quantity $\pi_k(n,\chi)$ entails studying the auxillary quantity
$$ M_z(n,\chi) := \sum_{f \in \mathcal{M}_n}\chi(f) z^{\Omega(f)} = \sum_{k\geq 0}z^k \pi_k(n,\chi)$$
and then recovering $\pi_k(n,\chi)$ from $M_z(n,\chi)$ using Cauchy's integral formula. Now $\chi(f) z^{\Omega(f)}$ is a multiplicative function of $f$ and has Dirichlet series with Euler product given by
$$G_z(u,\chi):=\sum_{f \in \mathcal{M}}\chi(f)z^{\Omega(f)}u^{\deg f} = \prod_{p \in \mathcal{P}}\left(1-z \chi(p) u^{\deg p}\right)^{-1}.$$
Since $|\mathcal{P}_n| \leq q^n$, this product converges absolutely in the range $|u| < \min\{|z|^{-1}, q^{-1}\}$. A direct application of the method makes use of the identity/definition
$$G_z(u,\chi) = L(u,\chi)^z F_z(u,\chi)$$
where
$$F_z(u,\chi):=\prod_{p \in\mathcal{P}}(1-z \chi(p) u^{\deg p})^{-1})(1-\chi(p) u^{\deg p})^z$$
and this product converges absolutely for $|u| < \min\{|z|^{-1}, q^{-1/2}\}.$
The crucial identity/definition we shall make use of is
\begin{equation}
G_z(u,\chi) = L(u,\chi)^{z}L(u^2,\chi^2)^{\frac{z(z-1)}{2}} E_{z}(u,\chi)
\end{equation}
where $$E_{z}(u,\chi) = \prod_{p \in\mathcal{P}}\left(1-z \chi(p) u^{\deg p}\right)^{-1}\left(1-\chi(p)u^{\deg p}\right)^{z}\left(1-\chi^2(p)u^{2 \deg p}\right)^{\frac{z(z-1)}{2}}$$
and this product converges absolutely for $|u| < \min\{|z|^{-1}, q^{-1/3}\}.$ This gives an explicit representation of the Dirichlet series for $\chi(f)z^{\Omega(f)}$ beyond the radius $|u| = q^{-1/2}$ and will allow us to extract an explicit formula for $M_z(n,\chi)$ in terms of the zeros of $L(u,\chi)$ which are on the circle $|u| = q^{-1/2}$.
The key point being that $E_z(u,\chi)$, and hence $G_z(u,\chi)$, is holomorphic for $|u|<\min\{|z|^{-1}, q^{-1/3} \}$. This follows because in that range
$$|1-\chi(p)zu^{\deg p}| \geq 1 - |zu^{\deg p}| \geq 1-|zu|>0$$ so the factors $(1-\chi(p)z u^{\deg p})^{-1}$ have no poles and
\begin{align*}
&\left(1-zT\right)^{-1}\left(1-T\right)^{z}\left(1-T^2\right)^{\frac{z(z-1)}{2}} \\
&= \left(1+zT+z^2T^2 + O(T^3)\right)\left(1-zT+\frac{z(z-1)}{2}T^2+ O(T^3)\right)\left(1-\frac{z(z-1)}{2}T^2+O(T^4)\right) \\
&= 1+ O(T^3)
\end{align*}
so each factor in the Euler product of $E_z(u,\chi)$ is $1+O(u^{3\deg p})$, hence the product converges absolutely and defines a holomorphic fucntion for $|u|<\min\{|z|^{-1}, q^{-1/3} \}$.

Before starting the proof, let's take a minute to be clear about what these expressions mean. Having factored $L(u,\chi)$ over its zeros $\rho$ as
$$L(u,\chi)=\prod_{\rho \: : \: L(\rho,\chi) = 0}(1-u/\rho)^{m_{\rho}}$$ as in equation (1), we define
$$L(u,\chi)^z:=\exp\left(z \sum_{\rho}m_{\rho}\log(1-u/\rho)\right)$$
where $\log$ is defined on the set $\mathbb{C}\backslash [0,-\infty)$ and takes real values on the positive reals and define $L(u^2,\chi^2)^{z(z-1)/2}$ similarly. These expressions define functions $L(u,\chi)^z$ and $L(u^2,\chi^2)^{z(z-1)/2}$, holomorphic in $u$ and $z$ for all $z \in \mathbb{C}$ and all $u \in\mathbb{C} \backslash \bigcup_{\rho}\{u \: : \: 1-u/\rho \in \mathbb{R}_{\leq 0}\}$.

The formula for $M_z(n,\chi)$ we need is given by Proposition~\ref{prop4.1} below. First though we prove a simple integral lemma.

\begin{lemma}\label{hanlem}
	Let $A,\delta>0$. Let $\mathcal{H}$ be the Hankel contour of radius 1 around 0 going along the negative real axis to $-n\delta$. Uniformly for $|z|\leq A$ we have
	$$\frac{1}{2\pi i}\int_{\mathcal{H}}w^{z}\frac{dw}{(1-w/n)^{n+1}} = \frac{1}{\Gamma(-z)} + O(1/n)$$
\end{lemma}
\begin{proof}
	We may suppose $n$ is sufficiently large. By Corollary 0.18 from~\cite{Ten} we have
	$$\frac{1}{2\pi i}\int_{\mathcal{H}}w^ze^w dw = \frac{1}{\Gamma(-z)} + O(e^{-n\delta/2})$$
	so it suffices to show that
	$$\int_{\mathcal{H}}\left|w^z\left(\frac{1}{(1-w/n)^{n+1}}-e^{w}\right)\right||dw| = O(1/n).$$
	On the region $\Re w > -n/2$ we have
	$$\frac{1}{(1-w/n)^{n+1}}-e^{w} = e^{-(n+1)\log(1-w/n)} -e^{w} = e^{w + O(1/n)} -e^{w} = O(e^{w}/n) $$
	and $\int_{\mathcal{H}}\left|w^z e^{w}\right||dw| = O(1).$ On the rest of the integral
	$$\int_{-n/2}^{-n\delta}\left|w^{z}\left(\frac{1}{(1-w/n)^{n+1}}-e^w\right) \right||dw| \ll n^{A+1}e^{- c n}$$ for some $c>0$ when $|z|\leq A$. This proves the lemma.
\end{proof}

\begin{prop}\label{prop4.1}
	For all $\epsilon>0$ the following holds uniformly for $|z|\leq q^{1/2-\epsilon}$. If $\chi^2 \neq \chi_0$,
	$$M_z(n,\chi) =  \sum_{\substack{\rho \: : \\ L(\rho, \chi) = 0 \\ |\rho|=q^{-1/2}}  }\rho^{-n}n^{-z m_\rho-1}\left\{\frac{F_z(\rho,\chi)c_\rho^z}{\Gamma(-z m_\rho)} + O\left(n^{-1} \right) \right\}$$
	and if $\chi^2=\chi_0$,
	\begin{multline*}
	M_z(n,\chi) =  \sum_{\substack{\rho \: : \\ L(\rho, \chi) = 0 \\ |\rho|=q^{-1/2} \\ \rho \neq \pm q^{-1/2}  }}  \rho^{-n}n^{-z m_\rho-1}\left\{\frac{F_z(\rho,\chi)c_\rho^z}{\Gamma(-z m_\rho)} + O\left(n^{-1} \right) \right\} \\
	+ \sum_{\pm} (\pm 1)^n q^{n/2}n^{-1 -zm_{\pm} +z(z-1)/2 }\left\{\frac{E_z(\pm q^{-1/2},\chi)c_{\pm}^z \left(\frac{\phi(d)}{2q^{\deg d}} \right)^{z(z-1)/2}}{\Gamma(-zm_{\pm}+z(z-1)/2)}  + O\left(n^{-1}\right) \right\}
	\end{multline*}
	for some constants $c_{\rho}$, $c_{\pm}$ defined in the course of the proof.
\end{prop}

\begin{proof}
	\begin{figure}[H]
		{
			\centering
			\includegraphics[scale=0.5]{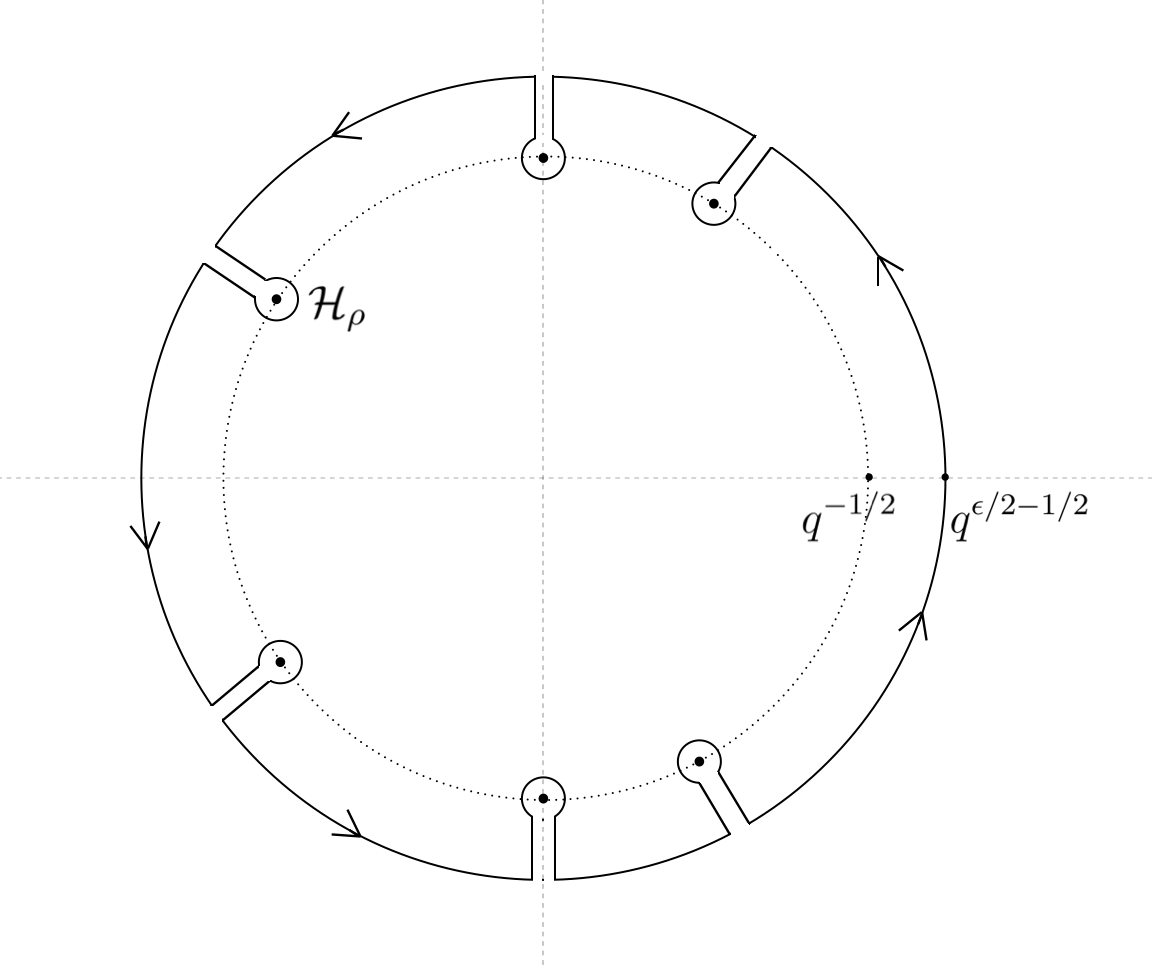}
			\caption{Contour of integration}
		}
	\end{figure}
	Applying Cauchy's integral formula with $r<q^{-1/2}$ gives
	\begin{equation}
	M_z(n,\chi) = \frac{1}{2\pi i}\int_{|u|=r}G_\chi(u,z) \frac{du}{u^{n+1}}.
	\end{equation}
	We shift this contour to write the left hand side in terms of the singularities of $G_z(u,\chi)$ with $|u|=q^{-1/2}$. If $\chi^2 \neq \chi_0$, these $\rho$ consist just of the zeros of $L(u,\chi)$. If $\chi^2 = \chi_0$ however, we also have to include the points $\pm q^{-1/2}$. For each such singularity $\rho$, let $\mathcal{H}_\rho$ be the contour that consists of a circle of radius $1/n$ traversed clockwise around $\rho$ and the two line segments on the ray from 0 to $\rho$ joining this small circle to the circle $|u|=q^{\epsilon/2-1/2}$. We may replace $\epsilon$ by $\min\{1/10,\epsilon\}$ if necessary to ensure $q^{\epsilon/2-1/2}<q^{-1/3}$.
Then (5) becomes
	$$
	\sum_{ \rho \in \mathcal{S}}\frac{1}{2\pi i}\int_{\mathcal{H}_\rho}G_z(u,\chi) \frac{du}{u^{n+1}} + O\left(\int_{|u|=q^{\epsilon/2 - 1/2}} \left| G_{z}(u,\chi)\frac{du}{u^{n+1}} \right| \right)
	=:\sum_{\rho \in \mathcal{S}} I_\rho + O\left(q^{n(\frac{1}{2}-\epsilon/2)}\right)
	$$
	Here we have used the fact that $G_z(u,\chi)$ is uniformly bounded in the region $|u|\leq q^{\epsilon/2-1/2}$, $|z|\leq q^{1/2-\epsilon}$. Let us now evaluate each of these Hankel contours, first the non-real $\rho$, then $\rho = \pm q^{-1/2}$.
	
	\hspace{-0.9cm}\textbf{Case 1:} Non-real $\rho$.
	
	Since $F_z(u,\chi)$ is holomorphic in $u$ on $\mathcal{H}_\rho$ we may use the approximation  $F_z(u, \chi) = F_z(\rho,\chi)+O(|u-\rho|)$ for the singularities $\rho = q^{-1/2}e^{-i\gamma_j} \neq \pm q^{-1/2}$ to get
	\begin{align*}
	I_{\rho}  =  F_z(\rho,\chi)\frac{1}{2\pi i}\int_{\mathcal{H}_{\rho}}L(u,\chi)^{z}\frac{du}{u^{n+1}}
	+ O\left(\int_{\mathcal{H}_{\rho}}|u-\rho||L(u,\chi)^z| \frac{|du|}{|u|^{n+1}} \right).
	\end{align*}
	In the error term, change variable to $w=n(1-u/\rho)$ so $u=\rho(1-w/n)$ to get
	\begin{align*}
	\int_{\mathcal{H}_{\rho}}\left|(1-u/\rho)^{1+zm_{\rho}}\left(\frac{L(u,\chi)}{(1-u/\rho)^{m_{\rho}}}\right)^z\right|\frac{|du|}{|u|^{n+1}}& \ll q^{n/2}|n^{-2-z m_{\rho}}|\int_{\mathcal{H}}|w^{1+z  m_{\rho}}|\frac{|dw|}{|(1-w/n)^{n+1}|} \\ & \ll q^{n/2}|n^{-2-zm_{\rho}}|
	\end{align*}
	since $|L(u,\chi)/(1-u/\rho)^{m_{\rho}}|$ is bounded above and below by positive constants on the contour of integration. Here $\mathcal{H}$ is the contour from Lemma 1 with $\delta = q^{\epsilon/2} - 1$.
	
	To evaluate the main term, again change to the variable $w=n(1-u/\rho)$ in the first integral above to get
	\begin{align*}
	\frac{\rho^{-n}n^{-1}}{2\pi i}\int_{\mathcal{H}}L(\rho(1-w/n),\chi)^z\frac{dw}{(1-w/n)^{n+1}}.
	\end{align*}
	Now if we define $c_\rho\neq 0$ by $L(u,\chi) = (1-u/\rho)^{m_\rho}(c_\rho + O(|1-u/\rho|))$, or equivalently, as  $c_{\rho}=L(u,\chi)/(1-u/\rho)^{m_{\rho}}\big|_{u=\rho}$ and $c_{\rho}^z=\exp(z \sum_{\rho'\neq \rho}m_{\rho'}\log(1-\rho/\rho'))$, then $$\left( \frac{L(u,\chi)}{(1-u/\rho)^{m_{\rho}}} \right)^z = c_{\rho}^z + O(|1-u/\rho|)$$
	and
	$$L(\rho(1-w/n),\chi)^z = (w/n)^{ z m_{\rho}}c_{\rho}^z + O(|w/n|^{z m_{\rho}+1}).$$
	By Lemma~\ref{hanlem}, this gives as main term
	$$\frac{c_\rho^z\rho^{-n}n^{-zm_\rho-1}}{2\pi i} \int_{\mathcal{H}}w^{z m_\rho}\frac{dw}{(1-w/n)^{n+1}} = \frac{ c_\rho^z\rho^{-n}n^{-z m_\rho-1}}{\Gamma(-z m_\rho)} + O(q^{n/2}|n^{-zm_{\rho}-2}|)$$
	and another error term bounded by
	$$\ll q^{n/2}|n^{-z m_{\rho}-2}|\int_{\mathcal{H}}|w^{z m_{\rho}+1}|\frac{|dw|}{|(1-w/n)^{n+1}|} \ll q^{n/2}|n^{-z m_{\rho}-2}|.$$

	\hspace{-0.9cm}\textbf{Case 2:} Real $\rho$.
	
	Now let's look at the singularities $\pm q^{-1/2}$. If $\chi^2$ is not principal, then exactly the same argument still works with $\rho = \pm q^{-1/2}$ and the convention $m_{\pm q^{-1/2}}=m_{\pm}$ because in that case $F_z(u, \chi) = L(u^2,\chi^2)E_z(u,\chi)$ is holomorphic near $u = \pm q^{-1/2}$.
	
	Suppose then that $\chi^2=\chi_0$ is principal. We now have to worry about the extra poles of $L(u^2,\chi_0)$ at $u=\pm q^{-1/2}$. First near $u = q^{-1/2}$, change to the variable $w = n(1-q^{1/2}u)$ so $u= q^{-1/2}(1-w/n)$. Then we want to evaluate
	$$\frac{q^{n/2}n^{-1}}{2\pi i}\int_{\mathcal{H}}L\left(\frac{1-w/n}{q^{1/2}},\chi\right)^z L\left(\frac{(1-w/n)^2}{q},\chi_0\right)^{\frac{z(z-1)}{2}}E_z\left(\frac{1-w/n}{q^{1/2}},\chi\right) \frac{dw}{(1-w/n)^{n+1}} $$
	Defining $c_{+} \neq 0 $ by $L(u,\chi)/(1-uq^{1/2})^{m_{+}} = c_{+} + O(|1-uq^{1/2}|)$, along the contour of integration we have
	$$
	L\left(\frac{1-w/n}{q^{1/2}},\chi\right)^z = (w/n)^{z m_{+}}c_{+}^z+O(|w/n|^{z m_{+} +1})
	$$
	and since $L(u, \chi_0)=\frac{1}{1-qu}\prod_{p|d}(1-u^{\deg p})$ and $\prod_{p|d}(1-q^{-\deg p}) = q^{-\deg d}\phi(d)$ we have $$L\left(\frac{(1-w/n)^2}{q},\chi_0\right)^{\frac{z(z-1)}{2}} = \left(\frac{n}{2w}\frac{\phi(d)}{q^{\deg d}}\right)^{\frac{z(z-1)}{2}} + O(|n/w|^{\frac{z(z-1)}{2}-1}).
	$$
	We also have,
	$$E_z\left(\frac{1-w/n}{q^{1/2}},\chi\right) = E_{z}(q^{-1/2},\chi)+ O(|w/n|)$$
	which together give the main term as
	\begin{multline*}
	q^{n/2}n^{-1 -zm_{+} +z(z-1)/2 }E_z(q^{-1/2},\chi)c_{+}^z \left(\frac{\phi(d)}{2q^{\deg d}} \right)^{z(z-1)/2}\frac{1}{2\pi i}\int_{\mathcal{H}}\frac{w^{z m_{+} - z(z-1)/2}}{(1-w/n)^{n+1}} dw \\
	= q^{n/2}n^{-1 -zm_{+} +z(z-1)/2 }\frac{E_z(q^{-1/2},\chi)c_{+}^z \left(\frac{\phi(d)}{2q^{\deg d}} \right)^{z(z-1)/2}}{\Gamma(-zm_{+}+z(z-1)/2)}
	+ O\left(q^{n/2}|n^{-zm_{+}+z(z-1)/2-2}|\right)
	\end{multline*}
	with the error terms all being $O(q^{n/2}|n^{-zm_{+}+z(z-1)/2-2}|).$
	
	The $-q^{-1/2}$ term is essentially the same. Together this proves Proposition~\ref{prop4.1}.
\end{proof}

\section{Saddle point lemmas} 
The following two lemmas allow us to deduce a formula for $\pi_k(n,\chi)$ from one for $M_z(n,\chi)$.
\begin{lemma}\label{l4.1}
	Let $A>0$ and let $p_n(z)=\sum_{k\geq 0}c_k(n)z^k$ be a sequence of polynomials such that uniformly for $n\geq 1$ and $|z|\leq A$
	\begin{equation}
	p_n(z)=n^{a z}\left(f(z) + O\left( n^{-1} \right)\right)
	\end{equation}
	for some real constant $a>0$ and some function $f(z)$ that is holomorphic on $\{z \in \mathbb{C} \: : \: |z|\leq A\}$. Then uniformly for $1\leq k \leq a A \log n,$
	$$c_k(n) = \frac{(a\log n)^k}{k!}\left(f\left(\frac{k}{a \log n}\right) + O\left(\frac{k}{(\log n)^2}\right)\right).$$
\end{lemma}
\begin{proof}
	By Cauchy's integral formula,
	$$c_k(n) = \frac{1}{2\pi i}\int_{|z|=r}p_n(z)\frac{dz}{z^{k+1}}.$$
	We choose $r$ to minimise the trivial bound
	$$|c_k(n)| \ll r\cdot\max_{|z|=r}\left|\frac{n^{a z}}{z^{k+1}}\right| = \frac{n^{a r}}{r^{k}}.$$
	Since $a>0$, the maximum occurs at $z = r = k/(a\log n)$. From (6) we have
	$$c_k(n) = \frac{1}{2 \pi i}\int_{|z|=r}n^{ az }f(z)\frac{dz}{z^{k+1}} + O\left(E\right)$$
	where $E = \frac{1}{n} \int_{|z|=r}\left|\frac{n^{az}}{z^{k+1}}\right||dz|.$
	Using
	$$\frac{1}{2 \pi i}\int_{|z|=r} (z-r)n^{a z}\frac{dz}{z^{k+1}} = \frac{(a \log n)^{k-1}}{(k-1)!} - r\frac{(a \log n)^k}{k!} = 0$$ and the approximation
	$$f(z) = f(r) + f'(r)(z-r) + O\left(|z-r|^2\right)$$ this becomes
	$$c_k(n) = f(r)\frac{(a \log n)^{k}}{k!} + O\left(r^{-k-1}\int_{|z|=r}\left|n^{a z}(z-r)^2\right||dz| + E \right).$$
	The integral in the error term is bounded by
	$$
	r^{3}\int_{-\pi}^{\pi}|1-e^{i\theta}|^2e^{k\cos \theta} d\theta \ll r^3\int_{-\infty}^{\infty}\theta^2 e^{k(1-\theta^2/5)}d\theta \ll r^3e^{k}k^{-3/2}
	$$
	which contributes an error of
	\begin{align*}
	( k/(a\log n))^{2-k}e^{k}k^{-3/2}
	\ll a^{k}(\log n)^{k-2}\frac{e^{k}\sqrt{k}}{k^{k}} \ll a^k\frac{k(\log n)^{k-2}}{k!}
	\end{align*}
	by Stirling's formula. The other term $E$ in the error contributes at most
	$$\ll \frac{r^{-k-1}}{n} \int_{|z|=r}\left|n^{a(z-r)}\right||dz| \ll \frac{r^{-k}}{n}\int_{-1/2}^{1/2}n^{-art^2}dt \ll\frac{1}{n}(a\log n)^{k} \frac{e^{k}}{k^{k+1/2}} \ll\frac{1}{n} \frac{(a \log n)^{k}}{k!}$$
	again by Stirling's formula.
\end{proof}
We also need the following quadratic variant of Lemma 2.
\begin{lemma}\label{l4.2}
	Let $A>0$ and let $p_n(z)=\sum_{k\geq 0}c_k(n)z^k$ be a sequence of polynomials such that uniformly for $n\geq 1$ and $|z|\leq A$
	\begin{equation}
	p_n(z)=n^{a z^2 + bz}\left(f(z) + O\left( n^{-1} \right)\right)
	\end{equation}
	for some real constants $a>0$ and $b>0$ and some function $f(z)$ that is holomorphic on $\{z \in \mathbb{C} \: : \: |z|\leq A\}$ with $f(0)=1$. Let $r$ be the positive root of the quadratic
	$$r^2+\frac{b}{2a}r - \frac{k}{2a \log n} = 0.$$ Then uniformly for $1\leq k \leq \min\{(\log n)^{1/2},bA \log n\}$ we have
	\begin{enumerate}[label=(\alph*)]
		\item $$c_k(n) = \frac{(b \log n)^k}{k!}\left(1 + O\left(\frac{k^{2}}{\log n}\right) \right).$$
		\vspace{0.3cm}
		\hspace{-1.5cm} In the range $1\leq k \leq \min\{(\log n)^{2/3}, bA \log n \}$ we have
		\item $$c_k(n) = \frac{(b \log n)^k}{k!}e^{\frac{a k^2}{b^2 \log n}}\left(1 + O\left(\frac{1}{k} + \frac{k^3}{(\log n)^{2}} \right)\right).$$
		\vspace{0.3cm}
		\hspace{-1.5cm} In the range $1\leq k \leq 2aA^2 \log n $ we have
		\item $$c_k(n) = \frac{n^{ar^2+br}}{r^k}\left(\frac{f(r)}{\sqrt{2\pi(4ar^2+br)\log n}} + O\left((r \log n)^{-3/2}\right)\right).$$
	\end{enumerate}
\end{lemma}

\begin{proof}
	Part (a) follows from the proof of Lemma~\ref{l4.1} with $r=k/(b \log n)$ and the approximation
	$$n^{a z^2} = 1+ O(k^2/\log n)$$ which holds for $|z| = r$. For part (c) we again use Cauchy's formula and the saddle point method. This time we choose $r$ to minimise the bound
	$$|c_k(n)| \ll r \cdot \max_{|z|=r}\left|\frac{n^{az^2+bz}}{z^{k+1}}\right| = \frac{n^{ar^2+br}}{r^k}.$$
	The maximum occurs at $z=r$ because $a>0$ and $b>0$. By differentiating, this is minimised when
	\begin{equation}
	r^2+\frac{b}{2a}r - \frac{k}{2a \log n} = 0, \hspace{1cm} \text{or} \hspace{1cm} r=\frac{b}{4a}\left(\sqrt{1+\frac{8ak}{b^2 \log n}}-1\right).
	\end{equation}
	Notice that in the range $k\leq 2aA^2 \log n$ we have
	$$r \leq \frac{b}{4a}\left(\sqrt{1+\frac{16a^2A^2}{b^2}}-1\right)\leq A$$
	since $\sqrt{1+x} \leq 1+\sqrt{x}$ for all $x>0$ so this is a valid choice for $r$.
	Now from (7) and
	$$f(z) = f(r) + f'(r) (z-r) + O\left(|z-r|^2\right)$$
	it follows that
	$$c_k(n) = f(r) I_1 + f'(r)I_2 + O(I_3) + O(E)$$
	where
	\begin{align*}
	I_1 &= \frac{1}{2 \pi i}\int_{|z|=r}\frac{n^{az^2+bz}}{z^{k+1}}dz \\
	I_2 &= \frac{1}{2 \pi i}\int_{|z|=r}\frac{n^{az^2+bz}}{z^{k+1}}(z-r)dz \\
	I_3 &= \int_{|z|=r}\left|\frac{n^{az^2+bz}}{z^{k+1}}(z-r)^2\right||dz| \\
	E &= \frac{n^{ar^2+br-1}}{r^k}.
	\end{align*}
	Writing $z=re^{2 \pi i t}$ and rearranging slightly this can we written
	$$c_k(n) = \frac{n^{ar^2 +br}}{r^k}\left\{ f(r)J_1 + r f'(r)J_2 + O(r^2 J_3) + O(1/n) \right\}$$
	where
	\begin{align*}
	J_1 &= \int_{-1/2}^{1/2}n^{(ar^2(e^{4 \pi i t}-1) + br(e^{2\pi i t}-1)}e^{-k 2\pi it}dt\\
	J_2 &=  \int_{-1/2}^{1/2}n^{(ar^2(e^{4 \pi i t}-1) + br(e^{2\pi i t}-1)}e^{-k 2\pi it}(1-e^{2\pi i t})dt\\
	J_3 &= \int_{-1/2}^{1/2}n^{(ar^2(\cos(4\pi t)-1) + br(\cos(2 \pi t)-1)}|1-e^{2\pi i t}|^2dt.
	\end{align*}
	The integrals $I_1$ and $I_2$ could be written explicitly as a sum of $k$ terms. Instead, we will asymptotically evaluate $J_1$ by expanding around the point $t=0$. This will give a main term and smaller error term provided $k \rightarrow \infty$. This is akin to approximating the integral $\frac{1}{2\pi i}\oint \frac{e^{z}}{z^{k+1}}dz=\frac{1}{k!}$ by $\frac{e^kk^{-k}}{\sqrt{2\pi k}}.$ For small values of $k$, part (a) is better.
	
	Let $\delta = (r \log n)^{-1/4}$. Using $e^{ix} = 1+ix-x^2/2-ix^3/6+O(x^4)$ for all $x$ and the definition of $r$ we have
	\begin{align*}J_1 &= \int_{-\delta}^{\delta}n^{ar^2(-8\pi^2t^2 -i32\pi^3t^3/3 +O(t^4) )+br(-2\pi^2 t^2 -i4\pi^3t^3/3 + O(t^4))}dt + O\left(n^{-br \delta^2}\right)\\
	&= \int_{-\delta}^{\delta}n^{-(4ar^2+br)2\pi^2t^2}(1+O(t^6(r\log n)^2 + t^4r \log n))dt + O\left(n^{-br \delta^2}\right) \\
	&= \frac{1}{\sqrt{2\pi (4 ar^2+br)\log n}} + O((r \log n)^{-3/2}).
	\end{align*}
	And in a similar vein,
	\begin{align*}
	J_2 &= \int_{-\delta}^{\delta}n^{ar^2(-8\pi^2t^2 -i32\pi^3t^3/3 +O(t^4) )+br(-2\pi^2 t^2 -i4\pi^3t^3/3 + O(t^4))}(-2\pi i t + O(t^2))dt  + O\left(n^{-br \delta^2}\right) \\
	&\ll \int_{-\delta}^{\delta}n^{-(4ar^2+br)2\pi^2t^2}(t^6(r \log n)^2 + t^4r \log n + t^2)dt + n^{-br \delta^2} \ll (r \log n)^{-3/2}
	\end{align*}
	and
	$$J_3 \ll \int_{-\infty}^{\infty}n^{-br t^2}t^2 dt \ll (r \log n)^{-3/2},$$
	which proves part (c).
	
	To prove part (b), we approximate $r$ and eliminate it from the expression given in part (c). Recall the definition of $r$
	\begin{equation}
	r = \frac{b}{4a}\left(\sqrt{1+\frac{8ak}{b^2 \log n}}-1\right) =\frac{k}{b \log n} - \frac{2ak^2}{b^3(\log n)^2} + O\left(\frac{k^3}{(\log n)^3}\right)
	\end{equation}
	so $k/2 \leq b r \log n \leq 2 k$ for $n$ sufficiently large and
	$$f(r) = f(0) + O(r) = 1 + O(k/\log n)$$
	and the expression from part (c) becomes
	$$\frac{n^{ar^2 + br}}{r^k}\frac{1}{\sqrt{2 \pi (4ar^2+br)\log n}}\left(1 + O \left(\frac{1}{k} + \frac{k}{\log n}\right)\right).$$
	Using the identity $ar^2+\frac{br}{2}-\frac{k}{2 \log n}=0$ and (7) the main term rearranges to
	\begin{multline*}
	\frac{n^{\frac{k}{2\log n}+\frac{br}{2}}r^{-k}}{\sqrt{2 \pi (2k-br \log n)}} = \frac{(b \log n)^ke^k e^{-\frac{ak^2}{b^2\log n} + O(\frac{k^3}{(\log n)^2})}}{\sqrt{2 \pi (k+\frac{2ak^2}{b^2 \log n}+O(\frac{k^3}{(\log n)^2})}}\left(k-\frac{2ak^2}{b^2 \log n} + O\big(\frac{k^3}{(\log n)^2}\big)\right)^{-k} \\
	= \frac{(b \log n)^ke^k}{k^k \sqrt{2 \pi k}} e^{-\frac{ak^2}{b^2\log n}}\left(1 + O(\frac{k}{\log n} + \frac{k^3}{(\log n)^2})\right)\left(1-\frac{2ak}{b^2 \log n} + O\big(\frac{k^2}{(\log n)^2}\big)\right)^{-k} \\
	=  \frac{(b \log n)^k}{k!} e^{\frac{ak^2}{b^2\log n}}\left(1 + O\big(\frac{1}{k} + \frac{k}{\log n} + \frac{k^3}{(\log n)^2}\big)\right).
	\end{multline*}
	Here we have used $(1 - x + O(y))^{-k} = e^{xk}(1+O(kx^2 + ky))$ when $kx^2, ky \ll 1$.
	Finally, we may leave out the $\frac{k}{\log n}$ term because $\frac{1}{k} + \frac{k^3}{(\log )^2} \geq \frac{k}{\log n}.$
\end{proof}

\section{Proofs of Theorems}

The proofs proceed by applying Lemmas~\ref{l4.1} and~\ref{l4.2} to Proposition~\ref{prop4.1}.

\subsection{Proof of Theorem~\ref{thm4.1}}
Suppose $\chi^2 \neq \chi_0$. We would like to apply Lemma~\ref{l4.1} to the polynomial
$$p_n(z)=\frac{1}{z}M_{-z}(n,\chi) =  - \sum_{f \in \mathcal{M}_n}\chi(f) (-z)^{\Omega(f)-1} = \sum_{k\geq 1}(-1)^k\pi_{k}(n,\chi)z^{k-1}$$
so that $(-1)^k\pi_{k}(n,\chi)$ is the coefficient of $z^{k-1}$. However, $p_n(z)$ isn't itself of the form required by Lemma~\ref{l4.1}. By Proposition~\ref{prop4.1}, it is a sum over $\rho$ of terms of the required form. But it is clear from the proof of Lemma~\ref{l4.1} that we can apply it to each summand separately which is what we shall do. So in our application of Lemma~\ref{l4.1} to the summand $\rho$ from Proposition~\ref{prop4.1}, we may take $A= q^{1/2-\epsilon}$ and have $a=m_\rho>0$ and
$$f(z)  = \frac{1}{z} \frac{F_{-z}(\rho, \chi)c_\rho^{-z}}{\Gamma(z m_\rho)} = m_\rho\frac{ F_{-z}(\rho, \chi) c_\rho^{-z}}{\Gamma(1+z m_\rho)}.$$
Then Theorem~\ref{thm4.1} follows after using
$$f((k-1)/(a \log n)) = f(0) + O(k/\log n) = m_\rho + O(k/(\log n)).$$

\subsection{Proof of Theorem~\ref{thm4.3}}
Suppose $\chi^2 = \chi_0$. With the same $p_n(z)$ as in the proof of Theorem~\ref{thm4.1}, this time we need Lemma~\ref{l4.1} and Lemma~\ref{l4.2} parts (a) and (b). Again, we don't apply these lemmas directly to $p_n(z)$, but it is clear from the proofs that we may apply them to each summand in Proposition 5 separately. For the $\rho \neq \pm q^{-1/2}$ terms we apply Lemma~\ref{l4.1} just as above. For the $\rho = \pm q^{-1/2}$ terms we apply Lemma~\ref{l4.2} part (a) in the range $1\leq k \leq (\log n)^{1/2}$ and part (b) in the range $1\leq k \leq (\log n)^{2/3}$ with $a = 1/2$, $b = m_{\pm} + 1/2$ and 
\begin{align*}
f(z) &=  \frac{E_{-z}(\pm q^{-1/2},\chi)c_{\pm}^{-z}\left(\frac{\phi(d)}{2q^{\deg d}}\right)^{z(z+1)/2}}{z\Gamma(zm_{\pm}+z(z+1)/2)} \\
&= (m_{\pm}+(z+1)/2) \frac{E_{-z}(\pm q^{-1/2},\chi)c_{\pm}^{-z}\left(\frac{\phi(d)}{2q^{\deg d}}\right)^{z(z+1)/2}}{\Gamma(1+zm_{\pm}+z(z+1)/2)}
\end{align*} so that $f(0) = m_{\pm} + 1/2$.

\subsection{Proof of Theorems~\ref{thm4.2} and~\ref{thm4.4}}

Suppose $m_j = 1$ for each $j$ and $m_{\pm} = 0$.
This time we apply Lemma~\ref{l4.1} and Lemma~\ref{l4.2} part (c) to the same $p_n(z)$ using Proposition~\ref{prop4.1}. For $\chi^2 \neq \chi_0$ and Theorem~\ref{thm4.2} we just apply the proof of Theorem~\ref{thm4.1} and the approximation
$$f((k-1)/(\log n)) = f(\alpha ) + o(1).$$
We see therefore that $h_j(\alpha) = F_{-\alpha}(\rho,\chi)c_{\rho}^{-\alpha}/\Gamma(1+\alpha)$ where $\rho = \alpha_j(\chi)^{-1}.$ Also, in the case that $m_{\rho}=1$, it follows from the definition of $c_{\rho}$ given in the proof of Proposition~\ref{prop4.1} that $c_{\rho} = -\rho L'(\rho, \chi) $.

For $\chi^2 = \chi_0$ and Theorem~\ref{thm4.4}, we have the two extra terms $\rho = \pm q^{-1/2}$. We can evaluate these with Lemma~\ref{l4.2} part (c) applied with $a = b = 1/2$ and
$$f_{\pm}(z)  =  \frac{E_{-z}(\pm q^{-1/2},\chi)c_{\pm}^{-z}\left(\frac{\phi(d)}{2q^{\deg d}}\right)^{z(z+1)/2}}{z\Gamma(z(z+1)/2)}.$$
Then $r>0$ satisfies
$$r^2+\frac{r}{2}-\frac{k-1}{\log n} = 0.$$ Since $k(n) \rightarrow \infty$ as $n \rightarrow \infty$ we also have $r \log n \rightarrow \infty$ as $n \rightarrow \infty$. Note that since $m_{\pm} = 0$, we have $c_{\pm } = L(\pm q^{-1/2},\chi)$. Therefore, using $\prod_{p|d}(1-q^{-\deg p}) = \phi(d)/q^{\deg d}$, one can check that $f_{\pm}(r) = h_{\pm}(\alpha)\sqrt{1+r^2/\alpha} + o(1)$ with $h_{\pm}(\alpha)$ defined as in the statement of Theorem 4. To finish the proof of Theorem~\ref{thm4.4} it therefore suffices to show that
\begin{equation}
\frac{n^{r^2/2+r/2}}{r^{k-1}}\frac{1}{\sqrt{2\pi(2r^2+r/2)\log n}} = (1+o(1))\frac{(\log n)^{k-1}}{(k-1)!}\frac{1}{\sqrt{1+r^2/\alpha}}n^{b((k-1)/\log n)}
\end{equation}
where $b$ is defined by (2) and check the stated conditions on the sign of $b$. This is a straightforward calculation using the definition of $r$ and Stirling's formula. The left hand side of~(10) is
\begin{multline*}
\frac{(\log n)^{k-1}}{(k-1)!}\frac{(k-1)!}{(r\log n)^{k-1}}\frac{n^{\frac{r}{4}+\frac{k-1}{2 \log n}}}{\sqrt{2\pi(k-1) (\frac{\log n}{k-1}r^2 + 1)}} \\
=  \frac{(\log n)^{k-1}}{(k-1)!}\frac{(k-1)!e^{k-1}}{(k-1 )^{k-1}}\frac{n^{\frac{r}{4}-\frac{k-1}{2\log n}+\frac{k-1}{\log n}\log(\frac{k-1}{r\log n})}}{\sqrt{2\pi(k-1) (\frac{\log n}{k-1}r^2 + 1)}} \\
=(1+o(1))\frac{(\log n)^{k-1}}{(k-1)!}\frac{n^{u\left(\frac{s-1}{2}-\log(2s)\right)}}{\sqrt{1+r^2/\alpha}}
\end{multline*}
where $u=(k-1)/\log n$ and $s=\frac{r}{2u}.$ Now $s$ is the positive root of the quadratic equation $$s^2+\frac{s}{4u}-\frac{1}{4u}=0$$
so
$$u = \frac{1-s}{4s^2}$$
and
$$s=\frac{1}{8u}\left(-1+\sqrt{1+16u}\right)$$
from which it follows that $0 \leq s \leq 1$. This proves (10) with $b$ defined by (3) since $\alpha = \lim_{n \rightarrow \infty} u$. Finally, it is easy to check the conditions on the sign of $b$ by noting that $\frac{s-1}{2}-\log(2s)$ is strictly decreasing on $(0,1)$ and equal to 0 at $s = \beta$ where $\beta$ is the unique solution to $\frac{\beta-1}{2}=\log (2\beta)$ with $0\leq \beta \leq 1$.

\section*{Acknowledgements}

We would like to thank Lucile Devin for helpful comments on an earlier draft of this paper and making available the paper~\cite{L-M}. This work was supported by the Engineering and Physical Sciences Research Council [EP/L015234/1] via the EPSRC Centre for Doctoral Training in Geometry and Number Theory (The London School of Geometry and Number Theory), University College London.

\end{document}